\documentclass[12pt,leqno]{article}
\usepackage{hyperref}
\usepackage{soul}
\usepackage[doc]{optional}
\usepackage{xcolor}
\definecolor{labelkey}{rgb}{0,0.08,0.45}
\definecolor{rekey}{rgb}{0,0.6,0.0}
\definecolor{Brown}{rgb}{0.45,0.0,0.05}
\usepackage{exscale,relsize}
\usepackage{amsmath}
\usepackage{amsfonts}
\usepackage{amssymb}
\usepackage{calc}
\usepackage{theorem}
\usepackage{pifont}      
\usepackage{graphicx}
\DeclareMathOperator{\weakstarly}{\rightharpoondown_{\mathrm{w*}}}

\newcommand{\wk}{\ensuremath{\operatorname{w*}}}

\newcommand{\scal}[2]{\langle{{#1},{#2}}\rangle}

\newcommand{\RR}{\ensuremath{\mathbb R}}

\newcommand{\RX}{\ensuremath{\,\left]-\infty,+\infty\right]}}

\newcommand{\NN}{\ensuremath{\mathbb N}}

\newcommand{\menge}[2]{\big\{{#1} \mid {#2}\big\}}

\newcommand{\To}{\ensuremath{\rightrightarrows}}

\newcommand{\dom}{\ensuremath{\operatorname{dom}}}

\newcommand{\gra}{\ensuremath{\operatorname{gra}}}

\newcommand{\inte}{\ensuremath{\operatorname{int}}}

\newcommand{\bd}{\ensuremath{\operatorname{bdry}}}
\newcommand{\ran}{\ensuremath{\operatorname{ran}}}
\newcommand{\rec}{\ensuremath{\operatorname{rec}}}

\newcommand{\conv}{\ensuremath{\operatorname{conv}}}

\renewcommand{\phi}{\ensuremath{\varphi}}

\newtheorem{theorem}{Theorem}[section]
\newtheorem{lemma}[theorem]{Lemma}
\newtheorem{fact}[theorem]{Fact}
\newtheorem{corollary}[theorem]{Corollary}
\newtheorem{proposition}[theorem]{Proposition}

\theoremstyle{plain}{\theorembodyfont{\rmfamily}
}
\theoremstyle{plain}{\theorembodyfont{\rmfamily}
}
\theoremstyle{plain}{\theorembodyfont{\rmfamily}
}
\theoremstyle{plain}{\theorembodyfont{\rmfamily}
\newtheorem{example}[theorem]{Example}}
\theoremstyle{plain}{\theorembodyfont{\rmfamily}
\newtheorem{remark}[theorem]{Remark}}

\theoremstyle{plain}{\theorembodyfont{\rmfamily}
}

\newcommand{\qede}{\hspace*{\fill}$\Diamond$\medskip}
\begin{document}


\title{\sffamily{Structure theory for maximally monotone operators with points of continuity}}

\author{
Jonathan M. Borwein\thanks{CARMA, University of Newcastle,
 Newcastle, New South Wales 2308, Australia. E-mail:
\texttt{jonathan.borwein@newcastle.edu.au}. Distinguished Professor
King Abdulaziz University, Jeddah.}\;
  and Liangjin\
Yao\thanks{CARMA, University of Newcastle,
 Newcastle, New South Wales 2308, Australia.
E-mail:  \texttt{liangjin.yao@newcastle.edu.au}.}}

\date{March 6,  2012}
\maketitle

\begin{abstract} \noindent
In this paper,
we consider the structure of maximally monotone operators in Banach space
 whose domains have nonempty interior and we
present new and explicit structure formulas for such operators. Along the way, we provide
 new proofs of the  norm-to-weak$^{*}$  closedness and of property (Q) for these operators (as recently proven by Voisei). Various applications and limiting examples are given.
\end{abstract}



\section{Introduction}

We assume throughout that
$X$ is a real Banach space with norm $\|\cdot\|$,
that $X^*$ is the continuous dual of $X$,
 and
that $X$ and $X^*$ are paired by $\scal{\cdot}{\cdot}$.
The \emph{closed unit ball} in $X$ is denoted by $B_X:=
\menge{x\in X}{\|x\|\leq1}$, $B_\delta(x):=x+\delta B_X$ (where $\delta>0$ and $x\in X$) and $\NN=\{1,2,3,\ldots\}$.

Let $A\colon X\To X^*$
be a \emph{set-valued operator} (also known as a relation, point-to-set mapping or multifunction)
from $X$ to $X^*$, i.e., for every $x\in X$, $Ax\subseteq X^*$,
and let
$\gra A := \menge{(x,x^*)\in X\times X^*}{x^*\in Ax}$ be
the \emph{graph} of $A$. The \emph{domain} of $A$
  is $\dom A:= \menge{x\in X}{Ax\neq\varnothing}$ and
$\ran A:=A(X)$ is the \emph{range} of $A$.

Recall that $A$ is
\emph{monotone} if
\begin{equation}
\scal{x-y}{x^*-y^*}\geq 0,\quad \forall (x,x^*)\in \gra A\;
\forall (y,y^*)\in\gra A,
\end{equation}
and \emph{maximally monotone} if $A$ is monotone and $A$ has
 no proper monotone extension
(in the sense of graph inclusion).
Let $A:X\rightrightarrows X^*$ be monotone and $(x,x^*)\in X\times X^*$.
 We say $(x,x^*)$ is \emph{monotonically related to}
$\gra A$ if
\begin{align*}
\langle x-y,x^*-y^*\rangle\geq0,\quad \forall (y,y^*)\in\gra
A.\end{align*}
Monotone operators have frequently proven to be a key class of objects in both
modern Optimization and Analysis; see, e.g., \cite{Bor1,Bor2,Bor3},
the books \cite{BC2011,
BorVan,BurIus,ph,Si,Si2,Rock70CA,RockWets,Zalinescu,Zeidler2A, Zeidler2B}
and the references given therein.

As much as  possible we adopt standard convex analysis notation.
Given a subset $C$ of $X$,
$\inte C$ is the \emph{interior} of $C$ and
$\overline{C}$ is   the
\emph{norm closure} of $C$.
 For the set $D\subseteq X^*$, $\overline{D}^{\wk}$
is the weak$^{*}$ closure of $D$, and  the norm $\times$ weak$^*$ closure of $C\times D$ is
$\overline{C\times D}^{\|\cdot\|\times\wk}$.
The \emph{indicator function} of $C$, written as $\iota_C$, is defined
at $x\in X$ by
\begin{align}
\iota_C (x):=\begin{cases}0,\,&\text{if $x\in C$;}\\
+\infty,\,&\text{otherwise}.\end{cases}\end{align} For every $x\in
X$, the \emph{normal cone} operator of $C$ at $x$ is defined by
$N_C(x):= \menge{x^*\in X^*}{\sup_{c\in C}\scal{c-x}{x^*}\leq 0}$, if
$x\in C$; and $N_C(x):=\varnothing$, if $x\notin C$;
 the \emph{tangent cone} operator of $C$ at $x$ is defined by
$T_C(x):= \menge{x\in X}{\sup_{x^*\in N_C(x)}\langle x, x^*\rangle\leq 0}$, if
$x\in C$; and $T_C(x):=\varnothing$, if $x\notin C$. The \emph{hypertangent cone} of $C$ at $x$, $H_C(x)$,
coincides with the interior of $T_C(x)$ (see \cite{BorStr,BorStr1}).

Let $f\colon X\to \RX$. Then
$\dom f:= f^{-1}(\RR)$ is the \emph{domain} of $f$.
We say $f$ is proper if $\dom f\neq\varnothing$.
Let $f$ be proper. The \emph{subdifferential} of
$f$ is defined by
   $$\partial f\colon X\To X^*\colon
   x\mapsto \{x^*\in X^*\mid(\forall y\in
X)\; \scal{y-x}{x^*} + f(x)\leq f(y)\}.$$

We say  a net $(a_{\alpha})_{\alpha\in\Gamma}$ in X is \emph{eventually bounded}
 if there exist $\alpha_0\in\Gamma$ and $M\geq0$ such that
\begin{align*}
\|a_{\alpha}\|\leq M,\quad \forall \alpha\succeq_\Gamma\alpha_0.
\end{align*}
We denote by $\longrightarrow$
and $\weakstarly$ respectively,
the norm convergence and weak$^*$ convergence
of  nets.

Let $A:X\To X^*$ be  monotone with $\dom A\neq\varnothing$ and consider a set $S\subseteq\dom A$.
 We define $A_S:X\To X^*$
by
\begin{align}
\gra A_{S}&=\overline{\gra A\cap(S\times X^*)}^{\|\cdot\|\times\wk}\nonumber\\
&=\big\{(x,x^*)\mid \exists\, \text{a net}\,
(x_{\alpha},x^*_{\alpha})_{\alpha\in\Gamma}\,
\text{in}\,\gra A\cap (S\times X^*)\, \text{such that}\, x_{\alpha}\longrightarrow x,
x^*_{\alpha}\weakstarly x^*\big\}\label{Deintcl}.
\end{align}
If $\inte\dom A\neq\varnothing$, we denote by
$A_{\inte}:=A_{\inte\dom A}$. We  note that $A_{\dom A} =A$ while $\gra A_S \subseteq \gra A_T$ for  $S \subseteq T$.

Let $A:X\rightrightarrows X^*$.
Following \cite{Hou}, we say $A$ has the upper-semicontinuity property \emph{property (Q)}  if
for every net $(x_{\alpha})_{\alpha\in J}$ in
$X$ such that $x_{\alpha}\longrightarrow x$, we have
\begin{align}\label{propQ}
\bigcap_{\alpha\in J}\overline{\conv\left[
\bigcup_{\beta\succeq_J\alpha} A (x_{\beta})\right]}^{\wk}
\subseteq Ax.
\end{align}

The remainder of this paper is organized as follows.
In Section~\ref{s:aux}, we collect preliminary results for future reference
and  the
reader's convenience. In Section~\ref{s:voi}, we study local boundedness properties
of monotone operators and
 also give a somewhat simpler proof of a recent result of Voisei \cite{VoiseiIn}.
The main result (Theorem~\ref{TheFom:2}) is proved
in Section~\ref{s:main}, and we also present a new proof
of a result of Auslender (Theorem~\ref{Ausmain:1}).
 A second structure theorem \ref{PropLe:2} --- which yields a strong version of property (Q)  for maximally monotone operators (Theorem~\ref{PropLe:2} ---) is also provided.
 Finally,  in Section~\ref{Appc:1} we present a few extra examples.

\section{Preliminary results}\label{s:aux}

We start with a classic compactness theorem.

\begin{fact}[Banach--Alaoglu]\label{BaAl}\emph{(See
  \cite[Theorem~2.6.18]{Megg} or \cite[Theorem~3.15]{Rudin}.)}
 The closed unit ball $B_{X^*}$ in $X^*$  is weak$^*$ compact.
\end{fact}

\begin{fact}[Rockafellar]\label{SubMR}\emph{(See \cite[Theorem~A]{Rock702},
 \cite[Theorem~3.2.8]{Zalinescu}, \cite[Theorem~18.7]{Si2} or \cite[Theorem~9.2.1]{BorVan}.)}
 Let $f:X\rightarrow\RX$ be a proper lower semicontinuous convex function.
Then $\partial f$ is maximally monotone.
\end{fact}

The prior result can fail in both incomplete normed spaces and in complete metrizable locally convex spaces \cite{BorVan}.
The next two important central results now has many proofs (see also \cite[Ch. 8]{BorVan}).

\begin{fact}[Rockafellar]\emph{(See \cite[Theorem~1]{Rock69} or  \cite[Theorem~2.28]{ph}.)}
\label{pheps:11}Let $A:X\To X^*$ be  monotone with
 $\inte\dom A\neq\varnothing$.
Then $A$ is locally bounded at $x\in\inte\dom A$,
 i.e., there exist $\delta>0$ and $K>0$ such that
\begin{align*}\sup_{y^*\in Ay}\|y^*\|\leq K,
\quad \forall y\in (x+\delta B_X)\cap \dom A.
\end{align*}
\end{fact}

\begin{fact}[Rockafellar]
\emph{(See \cite[Theorem~1]{Rock69} or \cite[Theorem~27.1 and Theorem~27.3]{Si2}.)}
\label{f:referee02c}
Let $A:X\To X^*$ be  maximal monotone with $\inte\dom A\neq\varnothing$. Then
$\inte\dom A=\inte\overline{\dom A}$ and $\overline{\dom A}$ is convex.
\end{fact}

The final two results we give are elementary.

\begin{fact}\emph{(\cite[Section~2, page~539]{BFG}.)}\label{normbn}
Let $A:X\rightrightarrows X^*$ be maximally monotone and a net
 $(a_{\alpha}, a^*_{\alpha})_{\alpha\in\Gamma}$ in $\gra A$.
Assume that $(a_{\alpha}, a^*_{\alpha})_{\alpha\in\Gamma}$
norm $\times$ weak$^*$ converges to $(x,x^*)$
  and $(a^*_{\alpha})_{\alpha\in\Gamma}$ is eventually bounded.
Then $(x,x^*)\in\gra A$.
\end{fact}

\begin{fact}
\emph{(See \cite[Proposition~4.1.7]{AE}.)}
\label{FTaC:1}
Let $C$ be a convex subset of $C$ with $\inte C\neq\varnothing$.
Then for every $x\in C$, $\inte T_C(x)=\bigcup_{\lambda>0}\lambda\left[\inte C-x\right]$.
\end{fact}

\section{Local boundedness properties}\label{s:voi}

The following result is extracted from part of the proof of \cite[Proposition~3.1]{Yao3}.
For the reader's convenience, we repeat the proof here.

\begin{fact}[Boundedness below]\label{extlem}
Let $A:X\rightrightarrows X^*$ be monotone and $x\in\inte\dom A$. Then there exist $\delta>0$ and  $M>0$ such that
 $x+\delta B_X\subseteq\dom A$ and $\sup_{a\in x+\delta B_X}\|Aa\|\leq M$.
Assume that $(z,z^*)$ is monotonically related to $\gra A$. Then
\begin{align}
\langle z-x, z^*\rangle
\geq \delta\|z^*\|-(\|z-x\|+\delta) M.
\end{align}

\end{fact}
\begin{proof}
Since $x\in\inte\dom A$, using Fact~\ref{pheps:11},
there exist $\delta>0$ and $M>0$ such that
\begin{align}
Aa\neq\varnothing \quad\text{and}\quad\sup_{a^*\in Aa}\|a^*\|\leq M,\quad \forall a\in (x+\delta
B_X).\label{RVT:10a}
\end{align}
Then we have
\begin{align}
&\langle z-x-b, z^*-b^*\rangle\geq0,\quad \forall b\in \delta B_X, b^*\in A(x+b)\nonumber\\
&\Rightarrow  \langle z-x, z^*\rangle-\langle b, z^*\rangle
+\langle z-x-b, -b^*\rangle\geq0,\quad\forall b\in \delta B_X, b^*\in  A(x+a)\nonumber\\
&\Rightarrow   \langle z-x, z^*\rangle-
\langle b, z^*\rangle\geq\langle z-x-b, b^*\rangle,\quad\forall b\in
\delta B_X, b^*\in A(x+b)\nonumber\\
&\Rightarrow \langle z-x, z^*\rangle-
\langle b, z^*\rangle\geq -(\|z-x\|+\delta) M, \quad\forall b\in \delta
B_X\quad\text{(by \eqref{RVT:10a})}\nonumber\\
&\Rightarrow \langle z-x, z^*\rangle
\geq \langle b, z^*\rangle -(\|z-x\|+\delta) M, \quad\forall b\in \delta
B_X.
\end{align}
Hence we have
\begin{align*}
\langle z-x, z^*\rangle
\geq \delta\|z^*\|-(\|z-x\|+\delta) M.
\end{align*}
\end{proof}

Fact \ref{extlem} leads naturally to the following result which has many precursors.

\begin{lemma}[Strong directional boundedness]\label{LemClC:1}
Let $A:X\rightrightarrows X^*$ be monotone and $x\in\inte\dom A$.
 Then there exist $\delta>0$ and  $M>0$ such that
 $x+2\delta B_X\subseteq\dom A$ and $\sup_{a\in x+2\delta B_X}\|Aa\|\leq M$.
Assume also that
$(x_0,x_0^*)$ is monotonically related to $\gra A$.
Then
 \begin{align*}\sup_{a\in\left[ x+\delta B_X,\,x_0\right[,\,
  a^*\in Aa}\|a^*\|\leq\frac{1}{\delta}\left(\|x_0-x\| +1\right)
\left(\|x^*_0\|+2M\right),
\end{align*}
where  $\left[x+\delta B_X,\,x_0\right[:=\big\{(1-t)y+tx_0\mid 0\leq t< 1, y\in x+\delta B_X\big\}$.
\end{lemma}
\begin{proof}
Since $x\in\inte\dom A$, by Fact~\ref{pheps:11},
there exist $\delta>0$ and $M>0$ such that
\begin{align}
x+2\delta
B_X\subseteq\dom A \quad\text{and}\quad\sup_{a^*\in Aa}\|a^*\|\leq M,\quad \forall a\in (x+2\delta
B_X).\label{ClCD:4}
\end{align}
Let $y\in x+\delta
B_X$. Then by \eqref{ClCD:4},
\begin{align}
y+\delta
B_X\subseteq\dom A \quad\text{and}\quad\sup_{a^*\in Aa}\|a^*\|\leq M,\quad \forall a\in (y+\delta
B_X).\label{ClCD:5}
\end{align}
Let $t\in\left[0,1\right[$ and $a^*\in A((1-t)y+tx_0)$.
By the assumption that $(x_0,x_0^*)$ is monotonically related to
$\gra A$, we have
\begin{align*}\big\langle a^*-x^*_0,(1-t)(y-x_0)\big\rangle=\big\langle a^*-x^*_0,(1-t)y+tx_0-x_0\big\rangle\geq0.
\end{align*}
Thus
\begin{align}
\langle a^*,x_0-y\rangle\leq\langle x_0-y,x^*_0\rangle.\label{ClCD:1}
\end{align}
By Fact~\ref{extlem} and \eqref{ClCD:5},
\begin{align}
\delta \|a^*\|&\leq \big\langle (1-t)y+tx_0-y,a^*\big\rangle+\big(\|(1-t)y+tx_0-y\|+\delta\big)M\nonumber\\
&\leq
\big\langle t(x_0-y),a^*\big\rangle+\big(\|x_0-y\|+\delta\big)M
\nonumber\\
&\leq \big\langle t(x_0-y),a^*\big\rangle+\big(\|x_0-x\|+2\delta\big)M\quad\text{(since $y\in x+\delta B_X$)}
.\label{ClCD:2}
\end{align}
Then by \eqref{ClCD:2} and \eqref{ClCD:1},
\begin{align}
\|a^*\|&\leq \frac{1}{\delta} t\langle x_0-y,x^*_0\rangle+\frac{M}{\delta}\|x_0-x\|+2M
\leq \frac{1}{\delta}\|x_0-y\|\cdot\|x_0^*\|+\frac{M}{\delta}\|x_0-x\|+2M
\nonumber\\
&\leq\frac{1}{\delta}\big(\|x_0-x\|+\delta)\|x_0^*\|+\frac{M}{\delta}\|x_0-x\|+2M
\quad\text{(since $y\in x+\delta B_X$)}\nonumber\\
&\leq \frac{1}{\delta}\|x_0-x\|\cdot\|x_0^*\|
+\|x^*_0\| +\frac{M}{\delta}\|x_0-x\|+2M\nonumber\\
&=\frac{1}{\delta}\|x_0-x\|\big(\|x_0^*\|+M\big)
+\|x^*_0\| +2M\nonumber\\
&\leq\frac{1}{\delta}\big(\|x_0-x\| +1\big)
\big(\|x^*_0\|+2M\big).\nonumber
\end{align}
Hence
\begin{align*}
\sup_{a\in\left[ x+\delta B_X,\,x_0\right[,\, a^*\in Aa}\|a^*\|
\leq\frac{1}{\delta}\left(\|x_0-x\| +1\right)
\left(\|x^*_0\|+2M\right).
\end{align*}
We now have the required estimate.
\end{proof}

The following result --- originally conjectured by the first author in \cite{Bor4}  --- was established by  Voisei in \cite[Theorem~37]{VoiseiIn} as part of a more complex set of results.
We next give a somewhat simpler proof by applying a similar technique to that
  used in   the proof of \cite[Prop~3.1, subcase 2]{Yao3}.

\begin{theorem}[Eventual boundedness]\label{bondednet:1}
Let $A:X\rightrightarrows X^*$ be  monotone
 such that $\inte\dom A\neq\varnothing$.
Then every norm $\times$ weak$^*$ convergent net in $\gra A$ is eventually bounded.
\end{theorem}

\begin{proof}
As the result and hypotheses are again invariant under translation, we can and do suppose that $0\in\inte\dom A$. Let
$(a_{\alpha}, a^*_{\alpha})_{\alpha\in\Gamma}$ in $\gra A$
be such that
\begin{align}
\text{$(a_{\alpha}, a^*_{\alpha})$
norm $\times$ weak$^*$ converges to $(x,x^*)$}.
 \label{NCW:00a}
 \end{align}
 Clearly, it suffices to show that $(a^*_{\alpha})_{\alpha\in\Gamma}$
  is eventually bounded. Suppose to
the contrary that $(a^*_{\alpha})_{\alpha\in\Gamma}$
is not eventually bounded. Then there exists a subnet of
$(a^*_{\alpha})_{\alpha\in\Gamma}$,  for convenience,
still denoted by $(a^*_{\alpha})_{\alpha\in\Gamma}$, such that
\begin{align}
\lim_{\alpha}\|a^*_{\alpha}\|=+\infty.\label{NWC:a01}
\end{align} We can and do suppose that $a^*_{\alpha}\neq0, \forall \alpha\in\Gamma$.
By Fact~\ref{extlem}, there exist $\delta>0$ and $M>0$ such that
\begin{align}
\langle a_{\alpha}, a^*_{\alpha}\rangle
\geq \delta\|a^*_{\alpha}\|-(\|a_{\alpha}\|+\delta) M,\quad\forall\alpha\in\Gamma.\label{NWC:a1}
\end{align}
Then we have
\begin{align}
\langle a_{\alpha}, \frac{a^*_{\alpha}}{\|a^*_{\alpha}\|}\rangle
\geq \delta-\frac{(\|a_{\alpha}\|+\delta) M}{\|a^*_{\alpha}\|},\quad\forall\alpha\in\Gamma.\label{NWCE:1}
\end{align}
By Fact~\ref{BaAl}, there  exists a weak* convergent \emph{subnet}
$(a^*_\beta)_{\beta\in I}$ of $(a^*_\alpha)_{\alpha\in\Gamma}$, say
\begin{align}\tfrac{a^*_\beta}{\|a^*_\beta\|}
\weakstarly  a_{\infty}^*\in X^*.\label{FCGG:9}\end{align}
Then taking the limit along the subset in \eqref{NWCE:1}, by \eqref{NCW:00a} and \eqref{NWC:a01},
 we have
\begin{align}
\langle x, a_{\infty}^*\rangle
\geq \delta.\label{NWC:1}
\end{align}
On the other hand, by \eqref{NCW:00a}, we have
\begin{align}
\langle x,a^*_{\alpha}\rangle\longrightarrow\langle x,x^*\rangle.\label{NWC:2}
\end{align}
Dividing by $\|a^*_{\alpha}\|$ in both sides of \eqref{NWC:2},
 then by \eqref{NWC:a01} and \eqref{FCGG:9}
 we take the limit along the subnet again to get
\begin{align}
\langle x, a^*_{\infty}\rangle=0.\label{NWC:3}
\end{align}
The above inequality contradict \eqref{NWC:1}.
Hence $(a_{\alpha}, a^*_{\alpha})_{\alpha\in\Gamma}$ is eventually bounded.
\end{proof}

\begin{corollary}[Norm-weak$^*$ closed graph] \label{CorCL:1}
Let $A:X\rightrightarrows X^*$ be maximally monotone
 such that $\inte\dom A\neq\varnothing$.
Then $\gra A$ is norm $\times$ weak$^*$ closed.
\end{corollary}

\begin{proof}
Apply Fact~\ref{normbn} and Theorem~\ref{bondednet:1}~.
\end{proof}

\begin{example}[Failure of graph to be norm-weak$^*$ closed]\label{ex:notclosed}
In \cite{BFG}, the authors showed that the statement of Corollary~\ref{CorCL:1}
cannot hold without the assumption of the nonempty interior domain even for the
subdifferential operators --- actually it  fails in the bw$^*$ topology. More precisely (see \cite{BFG} or
\cite[Example~21.5]{BC2011}):
 Let $f:\ell^2(\NN)\rightarrow\RX$ be defined by
 \begin{align}
 x\mapsto\max\big\{1+\langle x, e_1\rangle,\sup_{2\leq n\in\NN}
 \langle x, \sqrt{n}e_n\rangle\big\},\label{NormToEx:1}
 \end{align}
 where $e_n:=(0,\ldots,0,1,0,\cdots,0):$ the $n$th entry is $1$ and the others
are $0$.
 Then $f$ is proper lower semicontinuous and convex,
 but $\partial f$ is not norm $\times$ weak$^*$ closed.
 A more general construction in an infinite-dimensional
  Banach space $E$ is also given in \cite[Section~3]{BFG}. It is as follows:

  Let $Y$ be an infinite dimensional separable subspace of $E$, and $(v_n)_{n\in\NN}$ be
  a \emph{normalized Markushevich basis} of $Y$ with the dual coefficients $(v^*_n)_{n\in\NN}$.
  We defined $v_{p,m}$ and $v^*_{p,m}$ by
  \begin{align*}
  v_{p,m}:=\frac{1}{p}(v_p+v_{p^m})\quad\text{and}\quad
  v^*_{p,m}:=v_p^*+(p-1)v^*_{p^m},\quad m\in\NN,\, p\,\text{is prime}.
  \end{align*}
  Let $f: E\rightarrow\RX$ be defined by
 \begin{align}
 x\mapsto\iota_Y (x)+\max\big\{1+\langle x, v^*_1\rangle,\sup_{2\leq m\in\NN,\, p\,\text{is prime}}
 \langle x, v^*_{p,m}\rangle\big\}.
 \end{align}
  Then $f$ is proper lower semicontinuous and convex. We have that
$\partial f$ is not norm $\times$ bw$^*$ closed
 and hence $\partial f$ is not norm $\times$ weak$^*$ closed.
\qede
\end{example}

\begin{corollary}\label{CorCL:2}
Let $A:X\rightrightarrows X^*$ be maximally monotone with
$ \inte\dom A\neq\varnothing$. Assume that $S\subseteq\dom A$.
Then $\gra A_{S}\subseteq\gra A$ and in consequence
$\overline{\conv\left[A_{S}(x)\right]}^{\wk}\subseteq Ax, \forall x\in\dom A$.
Moreover, $Ax=A_S(x), \forall x\in S$ and hence $Ax=A_{\inte}(x), \forall x\in\inte\dom A$.
\end{corollary}
\begin{proof}
By \eqref{Deintcl} and Corollary~\ref{CorCL:1}, $\gra A_{S}\subseteq \gra A$.
Since $A$ is maximally monotone,
 (for every $x\in\dom A$), $Ax$ is convex and weak$^*$ closed. Thus $\overline{\conv\left[A_{S}(x)\right]}^{\wk}\subseteq Ax, \forall x\in\dom A$.
Let $x\in S$. Then by \eqref{Deintcl} again, $Ax\subseteq A_{S}(x)$ and hence
$Ax=A_{S}(x)$. Thus we have $A=A_{\inte}$ on $\inte\dom A$.
\end{proof}

We now turn to consequences of these  boundedness results.

\section{Structure of maximally monotone operators}\label{s:main}

A useful consequence of the Hahn-Banach separation principle \cite{BorVan} is:

\begin{proposition}\label{PrClC:4}
Let $D, F$ be nonempty subsets of $X^*$, and $C$
 be a  convex set of $X$ with $\inte C\neq\varnothing$.
 Assume that $x\in C$ and that for every $v\in \inte T_C(x)$,
\begin{align*}\sup\langle D, v\rangle\leq\sup\langle F, v\rangle<+\infty.
\end{align*}
Then
\begin{align}
D\subseteq \overline{\conv F+N_C(x)}^{\wk}.\label{PrClC:4E}
\end{align}
\end{proposition}

\begin{proof}
The separation principle ensures that suffices to show
\begin{align}
\sup\big\langle D,h\big\rangle\leq\sup
\big\langle N_{C}(x)+F, h\big\rangle,\quad\forall h\in X.
\label{TECLC:2}
\end{align}
We consider two cases.

\emph{Case 1}: $h\notin T_{C} (x)$. We have $\sup\big\langle N_{C}(x)+F, h\big\rangle
=+\infty$ since $\sup\big\langle N_{C}(x), h\big\rangle=+\infty$.
Hence \eqref{TECLC:2} holds.

\emph{Case 2}: $h\in T_{C}(x)$. Let $v\in\inte T_{C}(x)$.
Then (for every $t>0$) $h+tv\in\inte T_{C}(x)$
 by \cite[Fact~2.2(ii)]{BBC1}.
Now $z\mapsto\sup\big\langle D, z\big\rangle$ is lower semicontinuous,
and so by the assumption, we have
\begin{align*}
\sup\big\langle D, h\big\rangle&\leq\liminf_{t\rightarrow0^+}
\sup\big\langle D, h+tv\big\rangle
\leq \liminf_{t\rightarrow0^+}\sup\big\langle
 F, h+tv\big\rangle\\
&\leq\sup\big\langle F, h\big\rangle
+\liminf_{t\rightarrow0^+}t\sup\big\langle F, v\big\rangle\\
&=\sup\big\langle F, h\big\rangle\quad\text{( since $\sup\big\langle F, v\big\rangle$ is finite)}\\
&\leq\sup\big\langle N_{C}(x)+F, h\big\rangle.
\end{align*}
Hence \eqref{TECLC:2} holds and we have \eqref{PrClC:4E} holds.
\end{proof}

The proof of Proposition~\ref{PrClC:4} was inspired partially by
 that of \cite[Theorem~4.5]{BBC1}.

We can now provide our final technical proposition.

\begin{proposition}\label{LemClC:3}
Let $A:X\rightrightarrows X^*$ be maximally monotone
with $S \subseteq \inte\dom A\neq\varnothing$ such that $S$ is dense in $\inte\dom A$. Assume that
$x\in \dom A$ and  $v\in H_{\overline{\dom A}}(x) = \inte T_{\overline{\dom A}}(x)$.
Then there exists $x^*_0\in A_{S}(x)$ such that
\begin{align}
\sup\big\langle A_{S}(x), v\big\rangle=\big\langle x^*_0,v\big\rangle
=\sup\big\langle Ax,v\big\rangle.
\label{LETS:1}
\end{align}
In particular, $\dom A_S=\dom A$.
 \end{proposition}
\begin{proof}
By Corollary~\ref{CorCL:2}, $\gra A_{S}\subseteq\gra A$ and hence
\begin{align}
\sup\big\langle A_{S}(x), v\big\rangle\leq\sup\big\langle Ax,v\big\rangle.\label{LETS:2}
\end{align}
Now we show that
\begin{align}
\sup\big\langle A_{S}(x), v\big\rangle\geq\sup\big\langle Ax,v\big\rangle.\label{LETS:2b}
\end{align}
Appealing now to Fact~\ref{FTaC:1},  we can and do suppose that $v=x_0-x$,
where $x_0\in\inte\overline{\dom A}=\inte\dom A$ by Fact~\ref{f:referee02c}.
 Using Lemma \ref{LemClC:1} select $M,\delta>0$ such that $x_0+2\delta B_X\subseteq\dom A$ and
 \begin{align}\sup_{a\in\left[ x_0+\delta B_X,\,x\right[,\,
  a^*\in Aa}\|a^*\| \le M<+\infty.\label{estM}
\end{align}
Let $t\in\left]0,1\right[$. Then  by Fact~\ref{f:referee02c} again,
\begin{align}x+tB_\delta (v)=(1-t)x+tx_0 +t\delta B_X\subseteq \inte\overline{\dom A}
=\inte\dom A.\label{TECLC:3}
\end{align}
Then  by the monotonicity of $A$,
\begin{align}t\langle a^*-x^*,w\rangle=
\langle a^*-x^*,x+tw-x\rangle\geq0,\quad \forall
a^*\in A(x+tw),\, x^*\in Ax, w\in B_\delta(v).
\label{TECLC:25}
\end{align}
There exists a sequence $(x^*_n)_{n\in\NN}$ in $Ax$ such that
\begin{align}
\langle x^*_n, v\rangle\longrightarrow\sup\langle Ax,v\rangle.\label{TECLC:26}
\end{align}
Combining \eqref{TECLC:26} and \eqref{TECLC:25},
we have
\begin{align}\langle a^*-x_n^*,v+w-v\rangle\geq0,\quad \forall
a^*\in A(x+tw),\, w\in B_\delta(v),\,  n\in\NN.
\label{TECLC:27}
\end{align}
Fix $1<n\in\NN$. Thus, appealing to \eqref{estM} and \eqref{TECLC:27} yields,
\begin{align}\langle a^*,v\rangle&\geq\langle x_n^*,v\rangle-\langle a^*-x_n^*,w-v\rangle\nonumber\\
&\geq\langle x_n^*,v\rangle-(M+\|x_n^*\|)\cdot\|w-v\|
\quad \forall
a^*\in A(x+tw),\, w\in B_\delta(v),\,  n\in\NN.
\label{TECLC:28}
\end{align}
Take $\varepsilon_n:=\min\{\tfrac{1}{n(M+\|x_n^*\|)},\delta\}$ and $t_n:=\tfrac{1}{n}$.

Since $S$ is dense in $\inte \dom A$
 and $x+t_nB_{\varepsilon_n} (v)\subseteq\inte\dom A$ by \eqref{TECLC:3},
$ S\cap \left[x+t_nB_{\varepsilon_n} (v)\right]\neq\varnothing$.
  Then there exists $w_n\in X$ such that
\begin{align}
w_n \in B_{\varepsilon_n} (v),\quad\quad x+t_nw_n \in S\quad
\text{and then}\quad x+t_nw_n\longrightarrow x.\label{TECLC:29}
\end{align}
Hence, by \eqref{TECLC:28},
\begin{align}
\langle a^*, v\rangle\geq  \langle x_n^*,v\rangle -\frac{1}{n},\quad \forall
a^*\in A(x+t_nw_n).\label{ECLC:11}
\end{align}
Let $a_n^*\in A(x+t_nw_n)$. Then by \eqref{ECLC:11},
\begin{align}
\langle a^*_n, v\rangle\geq  \langle x_n^*,v\rangle -\frac{1}{n}.\label{TECLC:30}
\end{align}
By \eqref{estM} and \eqref{TECLC:3}, $(a_n^*)_{n\in\NN}$ is bounded.
Then by
Fact~\ref{BaAl}, there  exists a weak* convergent subnet of
$(a^*_\alpha)_{\alpha\in I}$ of $(a^*_n)_{n\in\NN}$ such that
\begin{align}
a_{\alpha}^*\weakstarly x_0^*\in X^*.\label{TECLC:31}
\end{align}
Then by \eqref{TECLC:29}, $x_0^*\in A_{S}(x)$ and thus by \eqref{TECLC:30},
\eqref{TECLC:31}
 and \eqref{TECLC:26}
\begin{align*}
\sup\big\langle A_{S}(x), v\big\rangle\geq\big\langle x^*_0,
 v\big\rangle\geq\sup\big\langle Ax,v\big\rangle.
\end{align*}
Hence  \eqref{LETS:2b} holds and  so does \eqref{LETS:1} by \eqref{LETS:2}.
The final conclusion then  follows from Corollary~\ref{CorCL:2} directly.
\end{proof}

An easy consequence is the reconstruction of $A$ on the interior of its domain. In the language of \cite{ph,PPN,BFK,BorZhu,BorVan} this is asserting the minimality of $A$ as a w$^*$-\emph{cusco}.

\begin{corollary}\label{CorCL:4}
Let $A:X\rightrightarrows X^*$ be maximally monotone with
$S \subseteq \inte\dom A\neq\varnothing$. for any  $S$  dense in $\inte\dom A$,
we have   $\overline{\conv\left[ A_S(x)\right]}^{\wk}=Ax=A_{\inte}(x), \forall x\in\inte\dom A$.
\end{corollary}
\begin{proof}
Corollary~\ref{CorCL:2} shows $\gra A_{S}\subseteq \gra A$.
Thus $A_S$ is monotone. By proposition~\ref{LemClC:3}, $A_S(x)\neq\varnothing$ on $\dom A$.
Then apply  \cite[Theorem~7.13 and Corollary~7.8]{ph} and  Corollary~\ref{CorCL:2} to obtain
\begin{align*}
\overline{\conv\left[A_S(x)\right]}^{\wk}=Ax
=A_{\inte}(x), \quad\forall x\in\inte\dom A,
\end{align*}
as required.
\end{proof}

There are many possible extensions of this sort of result along the lines studied in \cite{BFK}.
Applying Proposition~\ref{LemClC:3} and Lemma~\ref{LemClC:1}, we can also quickly recapture \cite[Theorem~2.1]{Aus93}.

\begin{theorem}[Directional boundedness in Euclidean space] \label{Ausmain:1}
Suppose that $X$ is finite-dimensional. Let $A:X\rightrightarrows X^*$ be maximally monotone
and $x\in\dom A$. Assume that there exist $d\in X$ and $\varepsilon_0>0$
such that $x+\varepsilon_0 d\in\inte\dom A$. Then
\[\left[Ax\right]_d:=\big\{x^*\in Ax \mid \langle x^*,d\rangle=\sup\langle Ax,d\rangle\big\}\]
is nonempty and compact. Moreover, if a sequence $(x_n)_{n\in\NN}$ in $\dom A$ is such that
$x_n\longrightarrow x$ and
\begin{align}
\lim\frac{x_n-x}{\|x_n-x\|}=d,\label{NWC:7}
\end{align}
 then for every $\varepsilon>0$, there exists $N\in\NN$ such that
 \begin{align}
 A(x_n)\subseteq \left[Ax\right]_d+\varepsilon B_{X^*},\quad \forall n\geq N.
 \label{NWC:8}
 \end{align}
\end{theorem}

\begin{proof}
By Fact~\ref{FTaC:1}, we have
$
d=\frac{1}{\varepsilon_0}(x+\varepsilon_0 d-x)
\in\frac{1}{\varepsilon_0}\left[\inte\dom A-x\right]\subseteq\inte T_{\overline{\dom A}}(x)$.
Then by Proposition~\ref{LemClC:3} and Corollary~\ref{CorCL:2}, there exists $v^*\in Ax$ such that
\begin{align}
\sup\langle Ax, d\rangle=\langle v^*,d\rangle. \label{NWC:9}
\end{align}
Hence $v^*\in\left[Ax\right]_d$ and thus $\left[Ax\right]_d \neq\varnothing$.

We next show that $\left[Ax\right]_d$ is compact.
Let $x^*\in \left[Ax\right]_d$.
By Fact~\ref{extlem},
there exist $\delta>0$ and $M>0$ such that
$
-\varepsilon_0\langle d, x^*\rangle=
\langle x-(x+\varepsilon_0 d), x^*\rangle
\geq \delta\|x^*\|-(\|\varepsilon_0 d\|+\delta) M$.
Then by \eqref{NWC:9},
$ \delta\|x^*\|\leq
(\|\varepsilon_0 d\|+\delta) M-\varepsilon_0\langle d, x^*\rangle
=(\|\varepsilon_0 d\|+\delta) M-\varepsilon_0\langle d, v^*\rangle
<+\infty$.
Hence $\left[Ax\right]_d$ is bounded. Clearly,  $\left[Ax\right]_d$ is closed and so
$\left[Ax\right]_d$ is compact.

Finally, we show that \eqref{NWC:8} holds.
By Lemma~\ref{LemClC:1} and $x+ \varepsilon_0 d\in\inte\dom A$, there exists $\delta_1>0$ such that
\begin{align}
\sup_{a\in\left[ x+ \varepsilon_0 d+{\delta_1} B_X,\,x\right[,\,
  a^*\in Aa}\|a^*\|<+\infty.\label{NWC:11}
\end{align}
By \eqref{NWC:7}, we have $\|d\|=1$ and  there exists $N\in\NN$ such that
for every $n\geq N$,
$
0<\|x_n-x\|<\varepsilon_0\quad\text{and }\quad
x_n\in x+\|x_n-x\| d+\|x_n-x\|\frac{\delta_1}{\varepsilon_0}B_X
\subseteq \left[ x+ \varepsilon_0 d+\delta_1 B_X,\,x\right[$.
Then by \eqref{NWC:11},
\begin{align}
\sup_{ a^*\in A(x_n),\,n\geq N}\|a^*\|<+\infty.\label{NWC:12}
\end{align}
Suppose to the contrary that \eqref{NWC:8} does not holds. Then there exists
$\varepsilon_1>0$ and  a subsequence $(x_{n,k})_{k\in\NN}$
of $(x_n)_{n\in\NN}$ and $x^*_{n,k}\in A(x_{n,k})$ such that
\begin{align}
x^*_{n,k}\notin \left[Ax\right]_d+\varepsilon_1 B_{X^*},\quad \forall k\in\NN.\label{NWC:13}
\end{align}
By \eqref{NWC:12},  there exists  a convergent subsequence of $(x^*_{n,k})_{k\in\NN}$,
for convenience,  still denoted by
$(x^*_{n,k})_{k\in\NN}$ such that
\begin{align}
x^*_{n,k}\longrightarrow x^*_{\infty}.\label{NWC:16}
\end{align}
Since $x_{n,k}\longrightarrow x$, by \eqref{NWC:16},
\begin{align}
(x, x^*_{\infty})\in\gra A.\label{NWC:17}
\end{align}
We claim that
\begin{align}
x^*_{\infty}\in \left[Ax\right]_d.\label{NWC:18}
\end{align}
By the monotonicity of $A$, recalling \eqref{NWC:9}, we have
$
\langle x^*_{n,k} -v^*, x_{n,k}-x\rangle\geq0,\quad\forall k\in\NN$.
Hence
\begin{align}
\langle x^*_{n,k} -v^*, \frac{x_{n,k}-x}{\|x_{n,k}-x\|}\rangle\geq0,\quad\forall k\in\NN.
\label{NWC:19}
\end{align}
Combining \eqref{NWC:16}, \eqref{NWC:7} and \eqref{NWC:19},
\begin{align}
\langle x^*_{\infty} -v^*, d\rangle\geq0.
\label{NWC:20}
\end{align}
By \eqref{NWC:9}, \eqref{NWC:20} and \eqref{NWC:17},
$x^*_{\infty}\in\left[Ax\right]_{d}$ and hence \eqref{NWC:18} holds.
Then  $x^*_{\infty}+\varepsilon_1 B_X\subseteq\left[Ax\right]_d+\varepsilon_1 B_X$
   and $x^*_{\infty}+\varepsilon_1 B_X$ contains infinitely many terms  of
    $(x^*_{n,k})_{k\in\NN}$, which contradicts \eqref{NWC:13}.
Hence, \eqref{NWC:8} holds as asserted.
\end{proof}

\begin{remark}
In the statement of \cite[Theorem~2.1]{Aus93}, the ``$x-x_n$"
in Eq (2.0) should be read as ``$x_n-x$". In his proof, the author considered it as
``$x_n -x$". \qede
\end{remark}

We next recall an alternate \emph{recession cone} description of $N_{\dom A}$. Consider
\begin{align}\label{def:rec}
{\rm  rec\,} A (x):=
\big\{x^*\in X^*\mid
\exists t_n\rightarrow 0^+, (a_n,a_n^*)\in\gra A \,
\text{such that}\,a_n\longrightarrow x,\, t_n a_n\weakstarly x^*\big\}.
\end{align}

\begin{proposition}[Recession cone]\label{Ncp:2}
Let $A:X\rightrightarrows X^*$ be monotone with $\gra A\neq\varnothing$. Then
for every $x\in\dom A$ one has
\begin{align*}
N_{\overline{\dom A}}(x)={\rm  rec \,} A (x).
\end{align*}

\end{proposition}
\begin{proof}
Let $x\in\dom A$.
We first show that
\begin{align}
{\rm  rec\,} A (x)\subseteq N_{\overline{\dom A}}(x).\label{Ncp:2E1}
\end{align}
Let $x^*\in {\rm  rec\,} A (x)$. Then  there exist $(t_n)_{n\in\NN}$ in $\RR$ and
$(a_n,a_n^*)_{n\in\NN}$ in $\gra A$ such that \begin{align}
t_n\longrightarrow 0^+,\, a_n\longrightarrow x\quad\text{and}\quad
t_n a^*_n\weakstarly x^*.\label{Ncp:2E2}
\end{align}
By \cite[Corollary~2.6.10]{Megg}, $(t_n a^*_n)_{n\in\NN}$ is bounded.
By the monotonicity of $A$,
\begin{align*}
\langle a_n-a, a_n^*\rangle\geq \langle a_n-a,a^*\rangle, \quad\forall (a,a^*)\in\gra A.
\end{align*}
Therefore,
\begin{align}
\langle a_n-a, t_na_n^*\rangle\geq t_n\langle a_n-a,a^*\rangle, \quad\forall (a,a^*)\in\gra A.
\label{Ncp:2E3}
\end{align}
Taking the limit in \eqref{Ncp:2E3}, by \eqref{Ncp:2E2}, we have
\begin{align*}
\langle x-a, x^*\rangle\geq 0, \quad\forall a\in\dom A.
\end{align*}
Thus, $x^*\in N_{\overline{\dom A}}(x)$. Hence \eqref{Ncp:2E1} holds.

It remains to  show that
\begin{align}
N_{\overline{\dom A}}(x)\subseteq {\rm  rec\,} A (x).\label{Ncp:2E4}
\end{align}
Let $y^*\in N_{\overline{\dom A}}(x)$ and $n\in\NN$.
Take $v^*\in Ax$. Since $A=N_{\overline{\dom A}}+A$,
we have $n y^*+ v^*\in Ax$.
Set $a_n:=x, a_n^*:=n y^*+v^*$ and $t_n:=\tfrac{1}{n}$.
Then we have
\begin{align*}
a_n\longrightarrow x, t_n\longrightarrow 0^+\quad\text{and}
\quad t_n a_n^*=y^*+\frac{1}{n}v^*\longrightarrow y^*.
\end{align*}
Hence $y^*\in {\rm  rec} A (x)$ and then \eqref{Ncp:2E4} holds.

Combining \eqref{Ncp:2E1} and \eqref{Ncp:2E4}, we have
$N_{\overline{\dom A}}(x)={\rm  rec\,} A (x)$.
\end{proof}

We are now ready for our main result, Theorem~\ref{TheFom:2}, the proof of which was inspired partially by
 that of \cite[Theorem~3.1]{Yao2}.

\begin{theorem}[Reconstruction of $A$, I]\label{TheFom:2}
Let $A:X\rightrightarrows X^*$ be maximally monotone with $S \subseteq \inte\dom A\neq\varnothing$ and  with $S$ dense in $\inte\dom A$.
Then
\begin{align}
Ax&=N_{\overline{\dom A}}(x)+\overline{\conv\left[A_{S}(x)\right]}^{\wk}
={\rm  rec\,} A(x)+\overline{\conv\left[A_{S}(x)\right]}^{\wk},\label{ThNp:S1}
 \quad \forall x\in X,
\end{align}
where ${\rm  rec\,} A(x)$ is as  in \eqref{def:rec}.
\end{theorem}
\begin{proof}
We first show that
\begin{align}
Ax=\overline{N_{\overline{\dom A}}(x)+\conv\left[A_{S}(x)\right]}^{\wk},
 \quad \forall x\in X.\label{Th:FIp1}
\end{align}
By Corollary~\ref{CorCL:2}, we have
$\conv\left[A_{S}(x)\right]\subseteq Ax,\ \forall x\in X$.
Since likewise $A=A+N_{\overline{\dom A}}$,
\begin{align}
\overline{N_{\overline{\dom A}}(x)+\conv\left[A_{S}(x)\right]}^{\wk}
\subseteq Ax,\quad \forall x\in X.\label{TEClC:a1}
\end{align}

It remains show that
\begin{align}
Ax\subseteq\overline{N_{\overline{\dom A}}(x)+\conv\left[A_{S}(x)\right]}^{\wk},
\quad \forall x\in \dom A.
\label{TEClC:1}
\end{align}
Let $x\in\dom A$.
By the maximal monotonicity of $A$ and  Proposition~\ref{LemClC:3}, both
$Ax$ and $A_{S}(x)$
  are nonempty sets.
  Then applying Proposition~\ref{PrClC:4} and Proposition~\ref{LemClC:3} directly,
  we have
 \eqref{TEClC:1} holds and hence \eqref{Th:FIp1} holds.

We must still show
\begin{align}
Ax=N_{\overline{\dom A}}(x)+\overline{\conv\left[A_{S}(x)\right]}^{\wk},
 \quad \forall x\in X.\label{Th:FIp2}
\end{align}
Now, for every two sets $C,D\subseteq X^*$,
we have $C+\overline{D}^{\wk}\subseteq \overline{C+D}^{\wk}$. Then
by \eqref{Th:FIp1}, it suffices to show that
\begin{align}
\overline{N_{\overline{\dom A}}(x)+\conv\left[A_{S}(x)\right]}^{\wk}\subseteq
N_{\overline{\dom A}}(x)+\overline{\conv\left[A_{S}(x)\right]}^{\wk},
\quad \forall x\in\dom A.\label{TECLC:9}
\end{align}
We again can and do suppose that $0\in\inte\dom A$ and $(0,0)\in\gra A$.
Let $x\in\dom A$ and $x^*\in\overline{N_{\overline{\dom A}}(x)+\conv\left[A_{S}(x)\right]}^{\wk}$.
Then there exists nets $(x^*_{\alpha})_{\alpha\in I}$ in $N_{\overline{\dom A}}(x)$ and
$(y^*_{\alpha})_{\alpha\in I}$ in $\conv\left[A_{S}(x)\right]$ such that
\begin{align}
x^*_{\alpha}+y^*_{\alpha}\weakstarly x^*.\label{TECLC:10}
\end{align}
Now we claim that
\begin{align}
(x^*_{\alpha})_{\alpha\in I} \, \text{is eventually bounded}.\label{TECLC:11}
\end{align}
Suppose to the contrary that $(x^*_{\alpha})_{\alpha\in I}$ is not eventually bounded.
Then there exists a subnet of
$(x^*_{\alpha})_{\alpha\in I}$,  for convenience,
still denoted by $(x^*_{\alpha})_{\alpha\in I}$, such that
\begin{align}
\lim_{\alpha}\|x^*_{\alpha}\|=+\infty.\label{TECLC:12}
\end{align} We can and do suppose that $x^*_{\alpha}\neq0, \forall \alpha\in I$.
By $0\in\inte\dom A$ and $x^*_{\alpha}\in N_{\overline{\dom A}}(x)$
 (for every $\alpha\in I$),  there exists $\delta>0$  such that $\delta B_X\subseteq
 \overline{\dom A}$ and hence we have
\begin{align}
\langle x, x^*_{\alpha}\rangle
\geq\sup_{b\in B_{X}}\langle x^*_{\alpha},\delta b\rangle=
 \delta\|x^*_{\alpha}\|.\label{TECLC:13}
\end{align}
Thence, we have
\begin{align}
\langle x, \frac{x^*_{\alpha}}{\|x^*_{\alpha}\|}\rangle
\geq \delta.\label{TECLC:14}
\end{align}
By Fact~\ref{BaAl}, there  exists a weak* convergent subnet
$(x^*_\beta)_{\beta\in \Gamma}$ of $(x^*_\alpha)_{\alpha\in I}$, say
\begin{align}\frac{x^*_\beta}{\|x^*_\beta\|}
\weakstarly  x_{\infty}^*\in X^*.\label{TECLC:15}\end{align}
Taking the limit along the subnet in \eqref{TECLC:14},  by \eqref{TECLC:15},
we have
\begin{align}
\langle x, x_{\infty}^*\rangle
\geq \delta.\label{TECLC:16}
\end{align}
By \eqref{TECLC:10} and \eqref{TECLC:12}, we have
\begin{align}
\frac{x^*_{\alpha}}{\|x^*_{\alpha}\|}+\frac{y^*_{\alpha}}{\|x^*_{\alpha}\|}\weakstarly 0.\label{TECLC:17}
\end{align}
And so by \eqref{TECLC:15},
\begin{align}
\frac{y^*_{\beta}}{\|x^*_{\beta}\|}\weakstarly -x_{\infty}^*.\label{TECLC:18}
\end{align}
By Corollary~\ref{CorCL:2}, $\conv\left[A_{S}(x)\right]\subseteq Ax$, and hence
$(y^*_{\alpha})_{\alpha\in I}$ is in $Ax$. Since $(0,0)\in\gra A$, we have
$\langle y^*_{\alpha},x\rangle\geq 0$ and so
\begin{align}
\big\langle \frac{y^*_{\beta}}{\|x^*_{\beta}\|},x\big\rangle\geq 0.\label{TECLC:19}
\end{align}
Using \eqref{TECLC:18} and taking the limit along the subnet in  \eqref{TECLC:19} we get
\begin{align}
\big\langle -x_{\infty}^*,x\big\rangle\geq 0,\label{TECLC:20}
\end{align}
which contradicts \eqref{TECLC:16}.
Hence, $(x^*_{\alpha})_{\alpha\in I}$ is eventually bounded and thus \eqref{TECLC:11} holds.

Then by Fact~\ref{BaAl} again, there exists a weak$^*$ convergent subset of
$(x^*_{\alpha})_{\alpha\in I}$,  for convenience, still denoted by $(x^*_{\alpha})_{\alpha\in I}$ which lies in the normal cone,
such that $x^*_{\alpha}\weakstarly w^*\in X^*$. Hence $w^*\in N_{\overline{\dom A}}(x)$
and $y^*_{\alpha}\weakstarly x^*-w^*\in\overline{\conv\left[ A_{S}(x)\right]}^{\wk}$ by \eqref{TECLC:10}.
Hence $x^*\in N_{\overline{\dom A}}(x)+\overline{\conv\left[ A_{S}(x)\right]}^{\wk}$ so that \eqref{TECLC:9}
holds.  Then we apply Proposition~\ref{Ncp:2} to get \eqref{ThNp:S1} directly.
\end{proof}

\begin{remark}
If $X$ is a \emph{weak Asplund space} (as holds if $X$ has a G\^{a}teaux smooth equivalent norm, see \cite{ph,PPN,BFK}),
the nets defined in $A_{S}$  in
Proposition~\ref{LemClC:3} and Theorem~\ref{TheFom:2} can be replaced by sequences.
By \cite[Chap. XIII, Notes and Remarks, page~239]{Diestel},
$B_{X^*}$ is weak$^*$ sequentially
compact. In fact, see  \cite[Chpt. 9]{BorVan}, this holds somewhat more generally.

Hence, throughout the proof of
Proposition~\ref{LemClC:3}, we can obtain weak$^*$
convergent subsequences instead of subnets. The rest of each subsequent argument is unchanged.
\qede
\end{remark}

In various classes of Banach space we can choose useful structure for $S\in S_A$, where
\begin{align*}
S_A:=\big\{
S\subseteq\inte\dom A\mid\text{$S$\, is dense in  $\inte\dom A$}\big\}.
\end{align*}

\begin{corollary}[Specification of $S_A$]\label{cor:cases} Let $A:X\rightrightarrows X^*$ be maximally monotone
with $\inte\dom A\neq\varnothing$. We may choose the dense set $S\in S_A$ to be as follows:
\begin{enumerate}
\item\label{Cor:Ca1} In a G\^{a}teaux smooth space, entirely within the residual set of non-$\sigma$ porous points of $\dom A$,
\item\label{Cor:Ca2} In an Asplund space, to include only a subset of the generic set points of single-valuedness and  norm to norm continuity of $A$,
\item\label{Cor:Ca3} In a separable Asplund space, to hold only  countably many angle-bounded points of $A$,
\item \label{Cor:Ca4}In  a weak Asplund space,  to include only a subset of the generic set of  points of single-valuedness (and norm to weak$^*$ continuity) of $A$,
    \item \label{Cor:Ca4b}In  a separable space,  to include  only  points of single-valuedness (and norm to weak$^*$ continuity) of $A$ whose complement is covered by a countable union of Lipschitz surfaces.
\item\label{Cor:Ca5} In finite dimensions, to include only points of differentiability of $A$ which are of full measure.
\end{enumerate}
\end{corollary}
\begin{proof} It suffices to determine in each case that the points of the given kind are dense.
\ref{Cor:Ca1}:  See \cite[Theorem 5.1]{Georg}.
\ref{Cor:Ca2}:  See \cite[Lemma~2.18 and Theorem 2.30]{ph}.
\ref{Cor:Ca3}:  See \cite[Theorem~2.19 and Theorem 2.11]{ph}.
\ref{Cor:Ca4}: See \cite[Proposition~1.1(iii) and Theorem~1.6]{PPN} or \cite[Theorem 4.31 and Example~7.2]{ph}.
\ref{Cor:Ca4b}:  See \cite{Ves1,Ves2}.
\ref{Cor:Ca5}: See \cite[Corollary~12.66(a)]{RockWets} or \cite[Exercise~9.1.1(2), page~412]{BorVan}.
\end{proof}.

These classes are sufficient but not necessary: for example, there are Asplund spaces with no equivalent G\^{a}teaux  smooth renorm \cite{BorVan}.
Note also that in \ref{Cor:Ca4b} and \ref{Cor:Ca5} we  also know that $A \diagdown S$ is a null set in the senses discussed \cite{CheZha}.

We now restrict attention to convex functions.

\begin{corollary}[Convex subgradients]\label{CorSubd}
Let $f:X\rightarrow\RX$ be proper lower semicontinuous and
 convex with $\inte\dom f\neq\varnothing$.  Let $S \subseteq \inte \dom f$ be given with $S$ dense in $\dom f$.
Then
\begin{align*}
\partial f(x)=N_{\overline{\dom f}}(x)
+\overline{\conv\left[{(\partial f)}_{S}(x)\right]}^{\wk}=N_{\dom f}(x)
+\overline{\conv\left[(\partial f)_{S}(x)\right]}^{\wk}, \quad \forall x\in X.
\end{align*}
\end{corollary}
\begin{proof} By \cite[Proposition~3.3 and Proposition~1.11]{ph},
 $\inte\dom\partial f\neq\varnothing$.
By the Br{\o}ndsted-Rockafellar Theorem
(see \cite[Theorem~3.17]{ph} or \cite[Theorem~3.1.2]{Zalinescu}),
 $\overline{\dom\partial f}=\overline{\dom f}$.
  Then  we may apply Fact~\ref{SubMR} and  Theorem~\ref{TheFom:2} to
  get (for every $x\in X$) $\partial f(x)=N_{\overline{\dom f}}(x)
+\overline{\conv\left[(\partial f)_{S}(x)\right]}^{\wk}$.
We have $ N_{\overline{\dom f}}(x)= N_{\dom f}(x), \forall x\in\dom\partial f$.
 Hence  $\partial f(x)=N_{\dom f}(x)
+\overline{\conv\left[(\partial f)_{S}(x)\right]}^{\wk},\forall x\in X
$.
\end{proof}

In this case Corollary \ref{cor:cases}  specifies settings in which only points of differentiability need be used (in \ref{Cor:Ca5} we recover Alexandroff's theorem on twice differentiability of convex functions), see \cite{BorVan} for more details.

\begin{remark}
Results closely related  to Corollary~\ref{CorSubd} have been obtained in \cite{Rock70CA,BBC1,JoTh, ThZag} and elsewhere. Interestingly, in the convex case we have obtained as much information more easily than by the direct convex analysis approach of \cite{BBC1}. \qede
\end{remark}

We finish this section by refining Corollary \ref{CorCL:4} and Theorem \ref{TheFom:2}.

Let $A:X\rightrightarrows X^*$. We define $\widehat{A}:X\rightrightarrows X^*$ by
\begin{align}
\gra \widehat{A}:=\big\{(x,x^*)\in X\times X^*\mid x^*\in\bigcap_{\varepsilon>0}
\overline{\conv \left[A(x+\varepsilon B_X)\right]}^{\wk}
\big\}.\label{NeEL:1}
\end{align}
Clearly, we have $\overline{\gra A}^{\|\cdot\|\times \wk}\subseteq\gra\widehat{A}$.

\begin{theorem}[Reconstruction of $A$, II]\label{PropLe:2}
Let $A:X\rightrightarrows X^*$ be maximally monotone
with $\inte\dom A\neq\varnothing$.
\begin{enumerate}\item \label{part:i}
 Then $\widehat{A}=A$.

In particular, $A$ has property (Q); and so has a norm $\times$  weak$^*$ closed graph.

\item  \label{part:ii} Moreover, if $S \subseteq \inte \dom A$ is dense in $\inte\dom A$ then
 \begin{align}
 \widehat{A_S}(x):&=\bigcap_{\varepsilon>0}
\overline{\conv \left[A(S \cap(x+\varepsilon B_X))\right]}^{\wk}
\supseteq\overline{\conv\left[A_{S}(x)\right]}^{\wk},\label{ThNp:S2}
 \quad \forall x\in X.
\end{align}
 Thence
  \begin{align}
Ax= \widehat{A_S}(x)+\rec A(x)
,\label{ThNp:S3}
 \quad \forall x\in X.
\end{align}
\end{enumerate}
 \end{theorem}
\begin{proof} \emph{Part \ref{part:i}.}
We first show that $\gra \widehat{A}\subseteq\gra A$.
Let $(x,x^*)\in\gra \widehat{A}$.
Now we show that $x\in\dom A$. We suppose that $0\in\inte\dom A$.
Since $x^*\in \overline{\conv \left[A(x+\tfrac{1}{n}B_X)\right]}^{\wk}$(for all $n\in\NN$),
\begin{align*}
\inf\big\langle A(x+\tfrac{1}{n}B_X),x \big\rangle&=\inf\big\langle \conv \left[A(x+\tfrac{1}{n}B_X)\right],x \big\rangle
=\inf\big\langle \overline{\conv \left[A(x+\tfrac{1}{n}B_X)\right]}^{\wk},x \big\rangle\\
&<\big\langle x, x^*\big\rangle+1.
\end{align*}
Then there exists
$z^*_n\in A(z_n)$ such that
\begin{align}\langle z^*_n,x\rangle\leq \langle x^*,x\rangle+1,\label{PLETS:a2}
\end{align}
 where $z_n\in x+\tfrac{1}{n}B_X$.
By Fact~\ref{extlem}, there exist $\delta_0>0$ and $M_0>0$ such that
\begin{align*}
&\delta_0\|z_n^*\|\leq\langle z_n, z_n^*\rangle+(\|z_n\|+\delta) M_0
=\langle z_n-x, z_n^*\rangle+\langle x, z_n^*\rangle+(\|z_n\|+\delta) M_0\\
&\leq \frac{1}{n}\|z_n^*\|
+\langle x^*,x\rangle+1+(\|x\|+1+\delta) M_0,\quad\forall n\in\NN\quad\text{(by \eqref{PLETS:a2})}.
\end{align*}
Hence $(z^*_n)_{n\in\NN}$ is bounded. By Fact~\ref{BaAl}, there exists a weak$^*$ convergent limit $z^*_{\infty}$ of
 a subnet of $(z^*_n)_{n\in\NN}$. Then $z_n\longrightarrow x$ and the maximal monotonicity of
 $A$, imply that $(x, z^*_{\infty})\in\gra A$ and so $x\in\dom A$.

Now let
 $v\in\inte T_{\overline{\dom A}}(x)$. We claim that
 \begin{align}
 \sup\big\langle \widehat{A}(x), v\big\rangle\leq\sup\big\langle Ax,v\big\rangle.\label{PLETS:2}
 \end{align}
 By Fact~\ref{FTaC:1},  we can and do suppose that $v=x_0-x$,
where $x_0\in\inte\overline{\dom A}=\inte\dom A$ by Fact~\ref{f:referee02c}.
There exists a sequence $(y^*_n)_{n\in\NN}$ in $\widehat{A}x$ such that
\begin{align}
\langle y^*_n, v\rangle\longrightarrow\sup\langle \widehat{A}x,v\rangle.\label{PLETS:3}
\end{align}
Using Lemma \ref{LemClC:1} select $M,\delta>0$ such that $x_0+2\delta B_X\subseteq\dom A$ and
 \begin{align}\sup_{a\in\left[ x_0+\delta B_X,\,x\right[,\,
  a^*\in Aa}\|a^*\| \le M<+\infty.\label{PestM}
\end{align}
 Then  by Fact~\ref{f:referee02c} again,
\begin{align}\left[ x_0+\delta B_X,\,x\right[\subseteq \inte\overline{\dom A}
=\inte\dom A.\label{PLETS:4}
\end{align}
Fix $\tfrac{1}{\delta}<n\in\NN$. Since $y^*_n\in \overline{\conv \left[A(x+\tfrac{1}{n}B_X)\right]}^{\wk}$, then
$\big\langle y^*_n, v\big\rangle\leq \sup \big\langle A(x+\tfrac{1}{n}B_X), v\big\rangle$.
Then there exist $x_n\in(x+\tfrac{1}{n}B_X)$ and
  $x^*_n\in A(x_n)$ such that
\begin{align}
\langle x^*_n,v\rangle\geq \langle y^*_n, v\rangle-\frac{1}{n}.\label{PLETS:a5}
\end{align}
Set $t_n:=\tfrac{1}{\delta\, n}$. Then,
\begin{align}
a_n:&=x_n +t_n v=x_n-x+x +t_n (x_0 -x)=t_n\left(x_0+\frac{x_n -x}{t_n}\right)+(1-t_n)x\nonumber\\
&
\in t_n(x_0+\delta B_X)+(1-t_n)x.\label{PLETS:5}
\end{align}
Select $a^*_n\in A (a_n)$ by \eqref{PLETS:4}.
Then  by the monotonicity of $A$,
$t_n\langle a_n^*-x_n^*,v\rangle=
\langle a_n^*-x_n^*,a_n-x_n\rangle\geq0$.
Hence $\langle a_n^*,v\rangle\geq
\langle x_n^*,v\rangle$.
Using \eqref{PLETS:a5},
we have
\begin{align}\langle a^*_n,v\rangle\geq \langle y^*_n, v\rangle-\frac{1}{n},
\quad \forall \tfrac{1}{\delta}<n\in\NN.
\label{PLETS:6}
\end{align}
 Thus, appealing to \eqref{PestM} and \eqref{PLETS:5} shows that
$(a^*_n)_{n\in\NN}$ is bounded.
Fact~\ref{BaAl}, now yields  a weak* convergent subnet of
$(a^*_\alpha)_{\alpha\in I}$ of $(a^*_n)_{n\in\NN}$ such that
\begin{align}
a_{\alpha}^*\weakstarly x_0^*\in X^*.\label{PLETS:7}
\end{align}
By Corollary~\ref{CorCL:1} and $a_n\longrightarrow x$,
we have $x_0^*\in Ax$. Combining \eqref{PLETS:6},
\eqref{PLETS:7} and \eqref{PLETS:3}, we obtain
\begin{align*}
\sup\big\langle Ax, v\big\rangle\geq\big\langle x^*_0,
 v\big\rangle\geq\sup\big\langle \widehat{A}x,v\big\rangle.
\end{align*}
Hence  \eqref{PLETS:2} holds.
Now applying Proposition~\ref{PrClC:4} and Proposition~\ref{LemClC:3},
we have
$\widehat{A}x\subseteq \overline{Ax+N_{\overline{\dom A}}(x)}^{\wk}=Ax$.
Hence $\gra \widehat{A}\subseteq\gra A$.

Since $\gra A\subseteq\gra \widehat{A}$, we have $\widehat{A}=A$. It is immediate $A$ has property (Q)
so has a norm $\times$  weak$^*$ closed graph.

\emph{Part \ref{part:ii}.} It only remains to prove \eqref{ThNp:S3}.   We first show that
\begin{align}
A_{S}(x)\subseteq\widehat{A_S}(x), \quad\forall x\in X.\label{PLETS:8}
\end{align}
By Proposition~\ref{LemClC:3}, $\dom A_S=\dom A$.
Let $w\in X$. If $w\notin\dom A$, then clearly,
 $A_S(w)\subseteq \widehat{A_S}(w)$.
Assume that $w\in\dom A$ and $w^*\in A_{S}(w)$. Then by \eqref{Deintcl},
there exist a net $(w_{\alpha}, w^*_{\alpha})_{\alpha\in I}$ in
$\gra A\cap (S\times X^*)$
 such that $w_{\alpha}\longrightarrow w$ and $w^*_{\alpha}\weakstarly w^*$.
 The for every $\varepsilon>0$, there exists $\alpha_0\in I$ such that
 $w_{\alpha}\in x+\varepsilon B_X,\quad\forall \alpha\succeq_I\alpha_0$. Thus
 \begin{align*}
 w_{\alpha}\in S\cap(w+\varepsilon B_X)\quad\text{and then}\quad
 w^*_{\alpha}\in A\big(S\cap(w+\varepsilon B_X)\big),\quad\forall \alpha\succeq_I\alpha_0.
 \end{align*}
Hence $w^*\in\overline{A\big(S\cap(w+\varepsilon B_X)\big)}^{\wk}\subseteq \overline{\conv \left[A\big(S\cap(w+\varepsilon B_X)\big)\right]}^{\wk}$ and thus
\eqref{PLETS:8} holds.

By \eqref{PLETS:8}, we have
\begin{align}
\overline{\conv\left[A_{S}(x)\right]}^{\wk}\subseteq
\widehat{A_S}(x),\quad\forall x\in X.\end{align}
Then by Proposition~\ref{Ncp:2},
\begin{align*}
\overline{\conv\left[A_{S}(x)\right]}^{\wk}+\rec A(x) \subseteq\widehat{A_S}(x)+\rec A(x)\subseteq Ax+\rec A(x)=Ax, \quad\forall x\in X.
\end{align*}
Thus, on appealing to Theorem~\ref{TheFom:2}, we obtain \eqref{ThNp:S3}.
\end{proof}

\begin{remark}
Property (Q,) first introduced by Cesari in Euclidean space, was recently established for maximally monotone operators with nonempty domain interior in Banach space by
 Voisei  in \cite[Theorem~42]{VoiseiIn}.
\end{remark}

\section{Final examples and applications}\label{Appc:1}

In general, we do not have  $Ax=\overline{\conv\left[A_{S}(x)\right]}^{\wk}, \forall x\in\dom A$, for a maximally monotone operator
$A:X\rightrightarrows X^*$ with
$S \subseteq \inte\dom A\neq\varnothing$ such that $S$ is dense in $\dom A$.

 We give a simple example to demonstrate this.

\begin{example}
Let $C$ be a closed convex subset of $X$ with $S\subseteq\inte C\neq\varnothing$ such that
$S$ is dense in $C$.
Then $N_C$ is maximally monotone  and $\gra (N_C)_{S}=C\times \{0\}$, but $N_C(x)\neq\overline{\conv \left[(N_C)_{S}(x)\right]}^{\wk},
\forall x\in \bd C$.  We have
$\bigcap_{\varepsilon>0}
\overline{\conv \left[N_C(x+\varepsilon B_X)\right]}^{\wk}=N_C(x),\,\forall x\in X$.
\end{example}

\begin{proof}
The maximal monotonicity of $N_C$ is directly from Fact~\ref{SubMR}. Since, for every $
x\in\inte C$, $N_C(x)=\{0\}$, $\gra (N_C)_{S}=C\times \{0\}$
 by \eqref{Deintcl} and Proposition~\ref{LemClC:3}.
Hence $\overline{\conv \left[(N_C)_{S}(x)\right]}^{\wk}=\{0\},\forall x\in C$. However,
$N_C(x)$ is unbounded, $\forall x\in\bd C$.
Hence $N_C(x)\neq\overline{\conv\left[(N_C)_{S}(x)\right]}^{\wk}, \forall x\in\bd C$.

By contrast, on applying Theorem~\ref{PropLe:2}, we have $\bigcap_{\varepsilon>0}
\overline{\conv \left[N_C(x+\varepsilon B_X)\right]}^{\wk}=N_C(x),\,\forall x\in X$.
\end{proof}

While the subdifferential operators in Example~\ref{ex:notclosed} necessarily
fail to have property (Q), it is possible for operators with no points of continuity to possess the property. Considering any closed linear mapping $A$ from a reflexive space $X$ to its dual,
we have $\widehat{A}=A$ and hence $A$ has property (Q). More generally:

\begin{example}
Suppose that $X$ is reflexive. Let $A:X\rightrightarrows X^*$ be such that
 $\gra A$ is nonempty closed and convex.
Then $\widehat{A}=A$ and hence $A$ has property (Q).
\end{example}
\begin{proof}
It suffices to show that $\gra\widehat{A}\subseteq \gra A$.  Let $(x,x^*)\in\gra \widehat{A}$.
Then we have
\begin{align*}
x^*\in \bigcap_{n\in\NN}\overline{\conv\left[A(x+\frac{1}{n}B_X)\right]}^{\wk}
=\bigcap_{n\in\NN}\overline{\conv\left[A(x+\frac{1}{n}B_X)\right]}=\bigcap_{n\in\NN}\overline{A(x+\frac{1}{n}B_X)}.
\end{align*}
Then there exists a sequence $(a_n, a^*_n)_{n\in\NN}$ in
$\gra A$ such that $a_n\longrightarrow x, a^*_n\longrightarrow
x^*$. The closedness of $\gra A$ implies that
$(x,x^*)\in\gra A$. Then $\gra \widehat{A}\subseteq\gra A$.
\end{proof}

It would be interesting to know whether $\widehat A$ and $A$ can differ for a maximal operator with norm $\times$ weak$^*$ closed graph.

Finally, we illustrate what Corollary \ref{CorSubd} says in the case of $x\mapsto \iota_{B_X}(x)+\frac{1}{p}\|x\|^p.$

\begin{example}\label{SForEx:1}
Let  $p> 1$ and $f:X\rightarrow\RX$ be defined by
\begin{align*}
x\mapsto \iota_{B_X}(x)+\frac{1}{p}\|x\|^p.
\end{align*}
Then for every $x\in \dom f$, we have
\begin{align}
N_{\dom f}(x)&=\begin{cases}\RR_+\cdot Jx,\, &\text{if}\, \|x\|=1;\\
\{0\},\, &\text{if}\, \|x\|<1
\end{cases}\label{SForEx1:e1}\\
(\partial f)_{\inte}(x)&=\begin{cases}\|x\|^{p-2}\cdot Jx,\, &\text{if}\, \|x\|\neq0;\\
\{0\},\, &\text{otherwise}
\end{cases}\label{SForEx1:e2}
\end{align}
where $J:=\partial\tfrac{1}{2}\|\cdot\|^2$ and $\RR_+:=\left[0,+\infty\right[$.
Moreover, $\partial f=N_{\dom f}+(\partial f)_{\inte}=N_{\dom f}+\partial \frac{1}{p}\|\cdot\|^p$,
and then $\partial f(x)\neq (\partial f)_{\inte}(x)=
\overline{\conv\left[(\partial f)_{\inte}(x)\right]}^{\wk}, \forall x\in\bd\dom f$.
We also have $\bigcap_{\varepsilon>0}
\overline{\conv \left[\partial f(x+\varepsilon B_X)\right]}^{\wk}=\partial f(x),\,\forall x\in X$.
\end{example}

\begin{proof}
By Fact~\ref{SubMR}, $\partial f$ is maximally monotone.
We have
\begin{align}
\partial f=\partial \frac{1}{p}\|\cdot\|^p,\quad\forall x\in\inte\dom \partial f.\label{SForEx1:e5}
\end{align}
 By \cite[Lemma~6.2]{BBC1},
\begin{align}
\partial \frac{1}{p}\|\cdot\|^p(x)&=\begin{cases}\|x\|^{p-2}\cdot Jx,\, &\text{if}\, \|x\|\neq0;\\
\{0\},\, &\text{otherwise}.
\end{cases}\label{SForEx1:e0a3}
\end{align}
Now we show that
\begin{align}
(\partial f)_{\inte}(x)=\partial \frac{1}{p}\|\cdot\|^p(x),\quad \forall x\in\dom f.\label{SForEx1:ea3}
\end{align}
Let $x\in\dom f$.
By Corollary~\ref{CorCL:1} and \eqref{SForEx1:e5}, we have
\begin{align}
(\partial f)_{\inte}(x)\subseteq\partial \frac{1}{p}\|\cdot\|^p(x).\label{SForEx1:e6}
\end{align}
Let $x^*\in\partial\frac{1}{p}\|\cdot\|^p(x)$. We first show that $(x,x^*)\in\gra(\partial f)_{\inte}$.
 If $\|x\|<1$,
 then $x\in\inte\dom f$ and hence by \eqref{SForEx1:e5}
and Corollary~\ref{CorCL:2},
$x^*\in\partial f(x)=(\partial f)_{\inte}(x)$.
Now we suppose that $\|x\|=1$. By \eqref{SForEx1:e0a3},
$x^*\in Jx$. Then $\tfrac{n-1}{n}x^*\in J(\tfrac{n-1}{n}x)$
 and hence $(\tfrac{n-1}{n})^{p-1}x^*\in
 \partial \frac{1}{p}\|\cdot\|^p(\tfrac{n-1}{n}x)$ by \eqref{SForEx1:e0a3},
 $\forall n\in\NN$.
 By \eqref{SForEx1:e5},
 \begin{align}
 (\tfrac{n-1}{n})^{p-1}x^*\in\partial f(\tfrac{n-1}{n}x),
 \quad \forall n\in\NN.\label{SForEx1:e7}
 \end{align}
 Since $0\in\inte\dom f$,
$\tfrac{n-1}{n}x\in\inte\dom f=\inte\dom\partial f, \forall n\in\NN$.
Since $\tfrac{n-1}{n}x\longrightarrow x, (\tfrac{n-1}{n})^{p-1}x^*\longrightarrow x^*$,
 by \eqref{SForEx1:e7}, $x^*\in (\partial f)_{\inte}(x)$.
 Hence $\partial\tfrac{1}{p}\|\cdot\|^p(x)\subseteq
(\partial f)_{\inte}(x)$. Thus by \eqref{SForEx1:e6},
we have \eqref{SForEx1:ea3} holds
and then we obtain \eqref{SForEx1:e2} by \eqref{SForEx1:e0a3}.

By \eqref{SForEx1:ea3},
\begin{align}(\partial f)_{\inte}(x)=
\overline{\conv \left[(\partial f)_{\inte}(x)\right]}^{\wk},\quad \forall x\in \dom f.
\label{SForEx1:e3}
\end{align}
On the other hand, since $N_{\dom f}=N_{B_X}$, 
 we can immediately get  \eqref{SForEx1:e1}.

Then by Corollary~\ref{CorSubd}, \eqref{SForEx1:e3} and \eqref{SForEx1:ea3}, we have
\begin{align}
\partial f(x)
=N_{\dom f}(x)+(\partial f)_{\inte}(x)=N_{\dom f}(x)+\partial \frac{1}{p}\|\cdot\|^p(x),
\quad\forall x\in X.
\label{SForEx1:e9}
\end{align}
Let $x\in\bd\dom f$. Then $\|x\|=1$.
On combining \eqref{SForEx1:e9}, \eqref{SForEx1:e1} and \eqref{SForEx1:e2},
\begin{align*}
\partial f(x)=\left[1,+\infty\right[\cdot Jx\neq Jx=(\partial f)_{\inte}(x)
=\overline{\conv \left[(\partial f)_{\inte}(x)\right]}^{\wk}.
\end{align*}
Theorem~\ref{PropLe:2} again implies that $\bigcap_{\varepsilon>0}
\overline{\conv \left[\partial f(x+\varepsilon B_X)\right]}^{\wk}=\partial f(x),\,\forall x\in X$.
\end{proof}

\section*{Acknowledgments}
The authors thank Dr.~Brailey Sims for his pertinent comments.
Jonathan  Borwein (and Liangjin Yao) was partially supported by various Australian Research  Council grants.



\begin{thebibliography}{99}

\bibitem{AE}
J.-P.\ Aubin and I.\ Ekeland, \emph{Applied Nonlinear Analysis}, John Wiley \& Sons Inc.,
New York, 1984.

\bibitem{Aus93}
A.\ Auslender,
``Convergence of stationary sequences for variational inequalities with maximal monotone",
\emph{
Applied Mathematics and Optimization},
vol.~28, pp.~161--172, 1993.

\bibitem{BBC1}
H.H.\ Bauschke, J.M.\ Borwein, and P.L.\ Combettes,
``Essential smoothness, essential strict convexity, and Legendre functions in Banach spaces",
\emph{Communications in Contemporary Mathematics}, vol.~3, pp.~615--647, 2001.

\bibitem{BC2011}
H.H.\ Bauschke and P.L.\ Combettes,
\emph{Convex Analysis and Monotone Operator Theory in Hilbert Spaces},
Springer, 2011.

\bibitem{Bor1}
J.M.\ Borwein,
``Maximal monotonicity via convex analysis'',
\emph{Journal of Convex Analysis}, vol.~13, pp.~561--586, 2006.

\bibitem{Bor2}J.M.\ Borwein, ``Maximality of sums of two maximal monotone operators
 in general
Banach space'',
\emph{Proceedings of the AMS}, vol.~135, pp.~3917--3924, 2007.

\bibitem{Bor4}J.M.\ Borwein, ``Asplund Decompositions of Monotone Operators",
in Proc. Control, Set-Valued Analysis and Applications, ESAIM: Proceedings,
 Alain Pietrus \& Michel H. Geoffroy, Editors, vol.~17 pp.~19--25, 2007.

\bibitem{Bor3}J.M.\ Borwein, ``Fifty years of maximal monotonicity'',
\emph{Optimization Letters}, vol.~4, pp.~473--490, 2010.


\bibitem{BFG}J.M.\ Borwein, S.\ Fitzpatrick, and R.\ Girgensohn,
``Subdifferentials whose graphs are not norm $\times$ weak$^*$ closed'',
\emph{Canadian Mathematical Bulletin}, vol.~4, pp.~538--545, 2003.


\bibitem{BFK}J.M.\ Borwein, S.\ Fitzpatrick, and P.\ Kenderov,
`Minimal convex uscos and monotone'',
\emph{Canadian Journal of Mathematics}, vol.~43, pp.461--476, 1991.


\bibitem{BorStr1}
J.M.\ Borwein and H.M.\ Strojwas,
``Directionally Lipschitzian mappings on Baire spaces",
\emph{Canadian Journal of Mathematics}, vol.~36, pp.~95--130, 1984.

\bibitem{BorStr}
J.M.\ Borwein and H.M.\ Strojwas, ``The hypertangent cone'',
 \emph{Nonlinear Analysis}, vol.~13, pp.~125--139, 1989.


\bibitem{BorVan}
J.M.\ Borwein and J.D.\ Vanderwerff,
\emph{Convex Functions},
Cambridge University Press, 2010.


\bibitem{BorZhu}
J.M.\ Borwein and Q.\ Zhu, ``Multifunctional and functional analytic methods in nonsmooth analysis,''
  in Nonlinear Analysis, Differential Equations and Control, F.H. Clarke and R.J. Stern (eds.) (NATO Advanced Study Institute, Montreal 1999), NATO Science Series C: vol.~528, Kluwer Academic Press, pp.~61--157,1999.

\bibitem{BurIus}
R.S.\ Burachik and A.N.\ Iusem,
\emph{Set-Valued Mappings and Enlargements of Monotone Operators},
Springer-Verlag, 2008.

\bibitem{CheZha}L.\ Cheng and W.\ Zhang,
``A note on non-support points, negligible sets,
G\^{a}teaux differentiability and Lipschitz embeddings,''
\emph{Journal of Mathematical Analysis and Applications}, vol.~350, pp.~531--536, 2009.

\bibitem{Diestel}
J.\ Diestel, \emph{Sequences and Series in Banach spaces}, Springer-Verlag, New York, 1984.


\bibitem{Georg}
P. Gr.~\ Georgiev,
``Porosity and differentiability in smooth Banach spaces",
\emph{Proceedings of the AMS}, vol.~133, pp.~1621--1628, 2005.


\bibitem{Hou}
S-H.\ Hou,
``On property (Q) and other semicontinuity properties of
multifunctions'',
\emph{Pacific Journal of Mathematics}, vol.~103, pp.~39--56, 1982.





\bibitem{JoTh}
A.\ Jofr\'{e} and L.\ Thibault,
``D-representation of subdiferentials of
directionally Lipschitz functions",
\emph{Proceedings of the AMS}, vol.~110, pp.~117--123, 1990.


\bibitem{Megg}
R.E.\  Megginson,
\emph{An Introduction to Banach Space Theory},
Springer-Verlag, 1998.









\bibitem{ph}
R.R.\ Phelps,
\emph{Convex Functions, Monotone Operators and
Differentiability},
2nd Edition, Springer-Verlag, 1993.



\bibitem{PPN}
D.~Preiss, R.R.\ Phelps, and I.\ Namioka,
``Smooth Banach spaces, weak Asplund spaces and monotone or usco mappings,
\emph{Israel Journal of Mathematics}, vol.~72,pp.~257--279, 1990.



\bibitem{Rock69}
R.T.\ Rockafellar,
``Local boundedness of nonlinear, monotone operators",
\emph{Michigan Mathematical Journal}, vol.~16, pp.~397--407, 1969.


\bibitem{Rock70CA}
 R.T.\ Rockafellar, \emph{Convex Analysis},
  Princeton Univ. Press, Princeton, 1970.

\bibitem{Rock702}
R.T.\ Rockafellar,
``On the maximality of sums of subdifferential mappings'',
\emph{Pacific Journal of Mathematics},
vol.~33, pp.~209--216, 1970.



\bibitem{RockWets}
R.T.\ Rockafellar and R.J-B Wets,
\emph{Variational Analysis}, 3rd Printing,
Springer-Verlag, 2009.


\bibitem{Rudin}
R.\ Rudin,
\emph{Functional Analysis},
Second Edition, McGraw-Hill, 1991.





\bibitem{Si}
S.\  Simons,
\emph{Minimax and Monotonicity},
Springer-Verlag, 1998.


\bibitem{Si2}
S.\ Simons, \emph{From Hahn-Banach to Monotonicity},
Springer-Verlag, 2008.




\bibitem{ThZag}
L.\ Thibault and D.\ Zagrodny,
``Integration of subdiferentials
of lower semicontinuous functions on Banach Spaces",
\emph{Journal of Mathematical Analysis and Applications}, vol.~189, pp.~33--58, 1995.

\bibitem{Ves1}
L.\, Vesel\'{y},
``On the multiplicity points of monotone operators of separable Banach spaces",
\emph{Commentationes Mathematicae Universitatis Carolinae}, vol.~27, pp.~551--570, 1986.

\bibitem{Ves2}
L.\, Vesel\'{y},
 ``On the multiplicity points of monotone operators on separable Banach spaces II",
  \emph{Commentationes Mathematicae Universitatis Carolinae}, vol.~28, pp.~295--299, 1987.

\bibitem{VoiseiIn}
M.D.\ Voisei,
``Characterizations and continuity properties for maximal monotone operators
with non-empty domain interior";
\texttt{http://arxiv.org/abs/1102.5055v1}, February 2011.


\bibitem{Yao3}
L.\ Yao,  ``The sum of a maximal monotone operator  of type (FPV) and a maximal monotone operator
with full domain is maximally monotone'', \emph{Nonlinear Analysis}, vol.~74, pp.~6144--6152, 2011.

\bibitem{Yao2}
L.\ Yao,  ``The sum of a maximally monotone
linear relation and the subdifferential of a proper lower semicontinuous
convex function is maximally monotone'',
\emph{Set-Valued and Variational Analysis}, in press.

\bibitem{Zalinescu}
{C.\ Z\u{a}linescu},
\emph{Convex Analysis in General Vector Spaces}, World Scientific
Publishing, 2002.



\bibitem{Zeidler2A}
E.\ Zeidler,
\emph{Nonlinear Functional Analysis and its Applications II/A:
Linear Monotone Operators},
Springer-Verlag, 1990.

\bibitem{Zeidler2B}
E.\ Zeidler,
\emph{Nonlinear Functional Analysis and its Applications II/B:
Nonlinear Monotone Operators},
Springer-Verlag,
1990.
\end{thebibliography}
\end{document}